\documentclass[a4, 12pt]{amsart}
\usepackage[mathscr]{eucal}
\usepackage{amssymb}
\usepackage{latexsym}
\usepackage{amsthm}
\theoremstyle{plain}
\newtheorem{theorem}{Theorem}[section]

\newtheorem{corollary}{Corollary}[section]
\newtheorem{remark}{Remark}[section]
\newtheorem{lemma}{Lemma}[section]

\newtheorem{definition}{Definition}[section]

\makeatletter
\@addtoreset{equation}{section}

\title[ Stability and area growth  of $\lambda$-hypersurfaces ]
{ Stability and area growth  of $\lambda$-hypersurfaces}
\author{Qing-Ming Cheng and  Guoxin Wei}
\address{Qing-Ming Cheng \\ Department of Applied Mathematics, Faculty of Sciences,
Fukuoka  University, 814-0180, Fukuoka,  Japan, cheng@fukuoka-u.ac.jp}
\address{Guoxin Wei \\  School of Mathematical Sciences, South China Normal University,
510631, Guangzhou,  China, weiguoxin@tsinghua.org.cn}

\begin{document}
\maketitle

\begin{abstract}
\noindent In this paper,  We define a $\mathcal{F}$-functional and study $\mathcal{F}$-stability of $\lambda$-hypersurfaces, which
extend a result of  Colding-Minicozzi \cite{[CM]}. Lower bound growth and upper bound growth of  area for
complete and non-compact $\lambda$-hypersurfaces are studied.

\end{abstract}




\section {Introduction}

\noindent
Let $X: M\rightarrow \mathbb{R}^{n+1}$ be a smooth $n$-dimensional immersed hypersurface in the $(n+1)$-dimensional
Euclidean space $\mathbb{R}^{n+1}$.
A  family $X(\cdot, t)$ of smooth immersions:
$$
X(\cdot, t):M\to  \mathbb{R}^{n+1}
$$
with  $X(\cdot, 0)=X(\cdot)$ is called  a mean curvature flow
if  they satisfy
\begin{equation*}
\dfrac{\partial X(p,t)}{\partial t}=\mathbf{H}(p,t),
\end{equation*}
where $\mathbf{H}(t)=\mathbf{H}(p,t)$ denotes the mean curvature vector  of hypersurface  $M_t=X(M^n,t)$ at point $X(p,t)$.
Huisken \cite{[H1]} proved that the mean curvature flow $M_t$ remains smooth and convex
until it becomes extinct at a point in the finite time.   If we rescale the flow about the point, the rescaling
converges to the round sphere. An immersed hypersurface $X:M\to  \mathbb{R}^{n+1}$ is called {\it a self-shrinker}  if
\begin{equation*}
H+\langle X,N\rangle=0,
\end{equation*}
where $H$ and $N$ denote the mean curvature and the  unit normal vector of $X:M\to  \mathbb{R}^{n+1}$, respectively.
 $\langle \cdot , \cdot\rangle $ denotes the standard inner product in $\mathbb{R}^{n+1}$.
It is  known that self-shrinkers play an important role in the study of the mean curvature flow
because they describe all possible blow ups at a given singularity of the mean curvature flow.

\noindent
Colding and Minicozzi \cite{[CM]}
have introduced a notation of $\mathcal{F}$-functional
and computed the first and the second variation formulas
of the $\mathcal{F}$-functional. They have proved that  an immersed hypersurface $X:M\to  \mathbb{R}^{n+1}$ is
a self-shrinker if and only if it is a critical point of the $\mathcal{F}$-functional.  Furthermore,
they have given a complete classification of the $\mathcal{F}$-stable complete self-shrinkers
with polynomial area growth.
\noindent
In \cite{[CW3]}, we consider a new type of mean curvature flow:
\begin{equation}\label{10-18-1}
\dfrac{\partial X(t)}{\partial t}=-\alpha(t) N(t) +\mathbf{H}(t),
\end{equation}
with
$$
\alpha(t) =\dfrac{\int_MH(t)\langle N(t), N\rangle e^{-\frac{|X|^2}2}d\mu}{\int_M\langle N(t), N\rangle e^{-\frac{|X|^2}2}d\mu},
$$
where  $N$ is the unit normal vector of $X:M\to  \mathbb{R}^{n+1}$.
We define {\it a weighted volume} of $M_t$ by
$$
V(t)=\int_M\langle X(t), N\rangle e^{-\frac{|X|^2}{2}}d\mu.
$$
We can prove that the flow \eqref{10-18-1} preserves the weighted volume $V(t)$. Hence, we call the flow \eqref{10-18-1}
{\it a weighted  volume-preserving mean curvature flow}.

\noindent
From a view of variations,
self-shrinkers of mean curvature flow can be characterized as critical points of the weighted area functional. In \cite{[CW3]},
the authors give  a definition of weighted volume and study  the weighted area functional for variations preserving this volume.
Critical points for
the weighted area functional for variations preserving this volume  are called  $\lambda$-hypersurfaces  by the authors in \cite{[CW3]}.
Precisely,   an $n$-dimensional hypersurface  $X:M\to  \mathbb{R}^{n+1}$ in Euclidean space  $ \mathbb{R}^{n+1}$ is  called {\it a $\lambda$-hypersurface}  if
\begin{equation}\label{eq:19-3-13-1}
\langle X, N\rangle +H=\lambda,
\end{equation}
where $\lambda$ is a constant,  $H$ and $N$ denote the mean curvature and  unit normal vector of $X:M\to  \mathbb{R}^{n+1}$, respectively.

\begin{remark}
If $\lambda=0$, $\langle X, N\rangle +H=\lambda=0$, then $X:M\to  \mathbb{R}^{n+1}$ is a self-shrinkers. Hence,  the notation of $\lambda$-hypersurfaces is a natural generalization of the self-shrinkers of the mean curvature flow. The equation \eqref{eq:19-3-13-1} also arises in the Gaussian isoperimetric problem.
\end{remark}

\noindent
In this paper,  we define
$\mathcal F$-functional. The first and second variation formulas of $\mathcal F$-functional
are given.  Notation of $\mathcal F$-stability and $\mathcal F$-unstability  of $\lambda$-hypersurfaces are
introduced. We prove that
spheres $S^n(r)$ with $r\leq \sqrt n$ or $r>\sqrt {n+1}$ are  $\mathcal F$-stable and
spheres $S^n(r)$ with $\sqrt n<r\leq \sqrt {n+1}$ are  $\mathcal F$-unstable.
In section 4, we study the  weak stability of  the weighted area functional for the weighted
volume-preserving variations.   In sections 5 and 6, the area growth
of complete and non-compact $\lambda$-hypersurfaces are studied.

\noindent
We should remark that this paper is the second part of our paper arXiv:1403.3177, which
is divided into two parts. The first part has been published \cite{CW1}.

\vskip 5mm

\section{The first variation of $\mathcal{F}$-functional}

\noindent
In this section, we will give another  variational characterization of $\lambda$-hypersurfaces.

\noindent
The following lemmas can be found in  \cite{[CW3]}.

\begin{lemma}\label{lemma 0}
If  $X: M\rightarrow \mathbb{R}^{n+1}$ is  a $\lambda$-hypersurface, then we have
\begin{equation}\label{eq:2001}
\aligned
 \mathcal{L}\langle X,a\rangle &=\lambda\langle N,a\rangle -\langle X,a\rangle,
\endaligned
\end{equation}

\begin{equation}\label{eq:2002}
\aligned
 \mathcal{L}\langle N,a\rangle &=-S\langle N,a\rangle,
\endaligned
\end{equation}
\begin{equation}\label{eq:2003}
\aligned
 \frac{1}{2}\mathcal{L}(|X|^2)&=n-|X|^2+\lambda\langle X,N\rangle ,
\endaligned
\end{equation}
where $\mathcal{L}$ is an elliptic operator given by $\mathcal{L}f=\Delta f-\langle X,\nabla f\rangle$,
 $\Delta$ and $\nabla$ denote the Laplacian and the gradient operator of the $\lambda$-hypersurface, respectively,
$a\in \mathbb{R}^{n+1}$ is constant vector, $S$ is the squared norm of the second fundamental form.
\end{lemma}

\begin{lemma}\label{lemma 1}
If $X: M\rightarrow \mathbb{R}^{n+1}$ is a hypersurface, $u$ is a $C^1$-function with compact support and
$v$ is a $C^2$-function, then
\begin{equation}
\int_M u(\mathcal{L}v)e^{-\frac{|X|^2}{2}}d\mu=-\int_M\langle \nabla u,\nabla v\rangle e^{-\frac{|X|^2}{2}}d\mu.
\end{equation}
\end{lemma}

\begin{corollary}\label{corollary 1}
Let $X: M\rightarrow \mathbb{R}^{n+1}$ be a complete hypersurface. If $u$, $v$ are $C^2$ functions satisfying
\begin{equation}
\int_M(|u\nabla v|+|\nabla u||\nabla v|+|u\mathcal{L}v|)e^{-\frac{|X|^2}{2}}d\mu< +\infty,
\end{equation}
then we have
\begin{equation}
\int_M u(\mathcal{L}v)e^{-\frac{|X|^2}{2}}d\mu=-\int_M\langle \nabla u,\nabla v\rangle e^{-\frac{|X|^2}{2}}d\mu.
\end{equation}
\end{corollary}

\begin{lemma}\label{lemma 10}
Let $X: M\rightarrow \mathbb{R}^{n+1}$ be an $n$-dimensional complete $\lambda$-hypersurface with polynomial area growth, then
\begin{equation}\label{eq:1001}
 \int_M (\langle X,a\rangle -\lambda\langle N,a\rangle )e^{-\frac{|X|^2}{2}}d\mu=0,
\end{equation}

\begin{equation}\label{eq:1003}
 \int_M \bigl(n-|X|^2+\lambda\langle X,N\rangle \bigl)e^{-\frac{|X|^2}{2}}d\mu=0,
\end{equation}

 \begin{equation}\label{eq:1010}
\aligned
&\ \ \ \int_M \langle X, a\rangle |X|^2e^{-\frac{|X|^2}{2}}d\mu\\
&=\int_M \biggl(2n\lambda\langle N, a\rangle +2\lambda\langle X,a\rangle (\lambda-H) -\lambda\langle N,a\rangle |X|^2\biggl)e^{-\frac{|X|^2}{2}}d\mu,
\endaligned
\end{equation}

 \begin{equation}\label{eq:1011}
\int_M\langle X,a\rangle ^2e^{-\frac{|X|^2}{2}}d\mu=\int_M\biggl(|a^{T}|^2+\lambda\langle N,a\rangle \langle X,a\rangle \biggl)e^{-\frac{|X|^2}{2}}d\mu,
\end{equation}
where $a^{T}=\sum_i<a,e_i>e_i$.

\begin{equation}\label{eq:1012}
\aligned
&\int_M\biggl(|X|^2-n-\frac{\lambda(\lambda-H)}{2}\biggl)^2e^{-\frac{|X|^2}{2}}d\mu\\
&=\int_M\biggl\{(\frac{\lambda^2}{4}-1)(\lambda-H)^2 +2n-H^2+\lambda^2\biggl\}e^{-\frac{|X|^2}{2}}d\mu.
\endaligned
\end{equation}

\end{lemma}

\noindent
Let $X(s): M\rightarrow \mathbb{R}^{n+1}$ a variation of $X$  with $X(0)=X$ and $\frac{\partial}{\partial s}X(s)|_{s=0}=fN$.
For $X_0\in \mathbb{R}^{n+1}$ and a real number $t_0 $,  {\it  $\mathcal{F}$-functional}  is defined by
\begin{equation*}
\aligned
&\ \ \ \ \mathcal{F}_{X_s,t_s}(s)=\mathcal{F}_{X_s,t_s}(X(s))\\
&=(4\pi t_s)^{-\frac{n}{2}}\int_M e^{-\frac{|X(s)-X_s|^2}{2t_s}}d\mu_s
  +\lambda (4\pi t_0)^{-\frac{n}{2}}(\frac{t_0}{t_s})^{\frac{1}{2}}\int_M \langle X(s)-X_s,N\rangle e^{-\frac{|X-X_0|^2}{2t_0}}d\mu,
\endaligned
\end{equation*}
where $X_s$ and $t_s$ denote variations of $X_0$ and $t_0$.
Let
\begin{equation*}
\frac{\partial t_s}{\partial s}=h(s),\ \ \frac{\partial X_s}{\partial s}=y(s),\ \
\frac{\partial X(s)}{\partial s}=f(s)N(s),
\end{equation*}
one calls that  $X: M\rightarrow \mathbb{R}^{n+1}$ is {\it a critical point of } $\mathcal{F}_{X_s,t_s}(s)$
if it is critical with respect to all normal variations and all variations in $X_0$ and $t_0$.

\begin{lemma}\label{lemma 3}
Let $X(s)$ be a variation of $X$ with normal variation vector field $\frac{\partial X(s)}{\partial s}|_{s=0}=fN$. If $X_s$ and $t_s$ are variations of $X_0$ and $t_0$ with $\frac{\partial X_s}{\partial s}|_{s=0}=y$ and $\frac{\partial t_s}{\partial s}|_{s=0}=h$, then the first variation formula of $\mathcal{F}_{X_s,t_s}(s)$ is given by
\begin{equation}\label{eq:1000}
\aligned
&\mathcal{F}^{'}_{X_0,t_0}(0)\\ &=(4\pi t_0)^{-\frac{n}{2}}\int_M \biggl(\lambda-(H+\langle \frac{X-X_0}{t_0},N\rangle )\biggl)f e^{-\frac{|X-X_0|^2}{2}}d\mu\\
&\ \ +(4\pi t_0)^{-\frac{n}{2}}\int_M \biggl(\langle \frac{X-X_0}{t_0},y\rangle -\lambda \langle N,y\rangle \biggl)e^{-\frac{|X-X_0|^2}{2}}d\mu\\
&\ \ +(4\pi t_0)^{-\frac{n}{2}}\int_M \biggl(\frac{|X-X_0|^2}{t_0}-n-\lambda\langle X-X_0,N\rangle \biggl)\frac{h}{2t_0}e^{-\frac{|X-X_0|^2}{2}}d\mu.
\endaligned
\end{equation}
\end{lemma}
\vskip 3pt\noindent {\it Proof}. Defining
\begin{equation}
\mathbb{A}(s)=\int_M e^{-\frac{|X(s)-X_s|^2}{2t_s}}d\mu_s, \ \
\mathbb{V}(s)=\int_M \langle X(s)-X_s,N\rangle e^{-\frac{|X-X_0|^2}{2t_0}}d\mu,
\end{equation}
then
\begin{equation*}
\aligned
\mathcal{F}^{'}_{X_s,t_s}(s)&=(4\pi t_s)^{-\frac{n}{2}}\mathbb{A}^{'}(s)+\lambda (4\pi t_0)^{-\frac{n}{2}}(\frac{t_0}{t_s})^{\frac{1}{2}}\mathbb{V}^{'}(s)\\
&\ \ -(4\pi t_s)^{-\frac{n}{2}}\frac{n}{2t_s}h \mathbb{A}(s)-\lambda (4\pi t_0)^{-\frac{n}{2}}(\frac{t_0}{t_s})^{\frac{1}{2}}
   \frac{h}{2t_s}\mathbb{V}(s).
\endaligned
\end{equation*}
Since
\begin{equation*}
\aligned
\mathbb{A}^{'}(s)&=\int_M \biggl\{-\langle \frac{X(s)-X_s}{t_s}, \frac{\partial X(s)}{\partial s}-\frac{\partial X_s}{\partial s}\rangle  +\frac{|X(s)-X_s|^2}{2t_s^2}h\\
&\ \ \ \ \ \ \ \ \ \ -H_s\langle \frac{\partial X(s)}{\partial s}, N(s)\rangle \biggl\}
    e^{-\frac{|X(s)-X_s|^2}{2t_s}}d\mu_s,
\endaligned
\end{equation*}

\begin{equation*}
\mathbb{V}^{'}(s)=\int_M\langle \frac{\partial X(s)}{\partial s}-\frac{\partial X_s}{\partial s}, N\rangle e^{-\frac{|X-X_0|^2}{2t_0}}d\mu,
\end{equation*}
we have
\begin{equation}\label{10-18-2}
\aligned
&\ \ \ \ \mathcal{F}^{'}_{X_s,t_s}(s)\\
&=(4\pi t_s)^{-\frac{n}{2}}\int_M -(H_s+\langle \frac{X(s)-X_s}{t_s}, N(s)\rangle )fe^{-\frac{|X(s)-X_s|^2}{2t_s}}d\mu_s\\
&\ \ \ +(4\pi t_0)^{-\frac{n}{2}}\sqrt{\frac{t_0}{t_s}}\int_M \lambda f\langle N(s),N\rangle e^{-\frac{|X-X_0|^2}{2t_0}}d\mu\\
&\ \ \ +(4\pi t_s)^{-\frac{n}{2}}\int_M \langle \frac{X(s)-X_s}{t_s},y\rangle e^{-\frac{|X(s)-X_s|^2}{2t_s}}d\mu_s\\
&\ \ \ +(4\pi t_0)^{-\frac{n}{2}}\sqrt{\frac{t_0}{t_s}}\int_M \lambda \langle -y,N\rangle e^{-\frac{|X-X_0|^2}{2t_0}}d\mu\\
&\ \ \ +(4\pi t_s)^{-\frac{n}{2}}\int_M (-\frac{n}{2t_s}+\frac{|X(s)-X_s|^2}{2t_s^2})he^{-\frac{|X(s)-X_s|^2}{2t_s}}d\mu_s\\
&\ \ \ +(4\pi t_0)^{-\frac{n}{2}}\sqrt{\frac{t_0}{t_s}}\int_M -\frac{h\lambda}{2t_s}\langle X(s)-X_s,N\rangle e^{-\frac{|X-X_0|^2}{2t_0}}d\mu.
\endaligned
\end{equation}
If $s=0$, then $X(0)=X$, $X_s=X_0$, $t_s=t_0$ and
\begin{equation*}
\aligned
&\ \ \ \ \mathcal{F}^{'}_{X_0,t_0}(0)\\
&=(4\pi t_0)^{-\frac{n}{2}}\int_M \biggl(\lambda-(H+\langle \frac{X-X_0}{t_0},N\rangle )\biggl)f e^{-\frac{|X-X_0|^2}{2}}d\mu\\
&\ \ \ +(4\pi t_0)^{-\frac{n}{2}}\int_M \biggl(\langle \frac{X-X_0}{t_0},y\rangle -\lambda \langle N,y\rangle \biggl)e^{-\frac{|X-X_0|^2}{2}}d\mu\\
&\ \ \ +(4\pi t_0)^{-\frac{n}{2}}\int_M \biggl(\frac{|X-X_0|^2}{t_0}-n-\lambda\langle X-X_0,N\rangle \biggl)\frac{h}{2t_0}e^{-\frac{|X-X_0|^2}{2}}d\mu.
\endaligned
\end{equation*}
$$\eqno{\Box}$$

\noindent
From Lemma \ref{lemma 3}, we know that if $X: M\rightarrow \mathbb{R}^{n+1}$ is a critical point of $\mathcal{F}$-functional
$\mathcal{F}_{X_s,t_s}(s)$, then
\begin{equation*}
H+\langle \frac{X-X_0}{t_0},N\rangle =\lambda.
\end{equation*}
We next prove that if $H+\langle \frac{X-X_0}{t_0},N\rangle =\lambda$, then $X: M\rightarrow \mathbb{R}^{n+1}$ must be a critical point of $\mathcal{F}$-functional $\mathcal{F}_{X_s,t_s}(s)$. For simplicity, we only consider the case of $X_0=0$ and $t_0=1$. In this case, $H+\langle \frac{X-X_0}{t_0},N\rangle =\lambda$ becomes
\begin{equation}
H+\langle X,N\rangle =\lambda.
\end{equation}

\noindent
Furthermore, we  know  that $X: M\rightarrow \mathbb{R}^{n+1}$ is a critical point of the $\mathcal{F}$-functional
$\mathcal{F}_{X_s,t_s}(s)$ if and only if $X: M\rightarrow \mathbb{R}^{n+1}$  is a  critical point of
$\mathcal{F}$-functional $\mathcal{F}_{X_0,t_0}(s)$ with respect to fixed $X_0$ and $t_0$.

\begin{theorem}\label{theorem 4}
$X: M\rightarrow \mathbb{R}^{n+1}$ is a critical point of $\mathcal{F}_{X_s,t_s}(s)$ if and only if
$$
H+\langle \frac{X-X_0}{t_0},N\rangle =\lambda.
$$
\end{theorem}
\begin{proof}
We only prove the result for $X_0=0$ and $t_0=1$. In this case, the first variation formula \eqref{eq:1000} becomes
\begin{equation}\label{eq:1004}
\aligned
\mathcal{F}^{'}_{0,1}(0)&=(4\pi)^{-\frac{n}{2}}\int_M \biggl(\lambda-(H+\langle X, N\rangle )\biggl)f e^{-\frac{|X|^2}{2}}d\mu\\
&\ \ +(4\pi)^{-\frac{n}{2}}\int_M \biggl(\langle X,y\rangle -\lambda \langle N,y\rangle \biggl)e^{-\frac{|X|^2}{2}}d\mu\\
&\ \ +(4\pi)^{-\frac{n}{2}}\int_M \biggl(|X|^2-n-\lambda\langle X,N\rangle \biggl)\frac{h}{2}e^{-\frac{|X|^2}{2}}d\mu.
\endaligned
\end{equation}
If $X: M\rightarrow \mathbb{R}^{n+1}$ is a critical point of $\mathcal{F}_{0,1}$, then $X: M\rightarrow \mathbb{R}^{n+1}$
should satisfy $H+\langle X, N\rangle =\lambda$. Conversely, if  $H+\langle X, N\rangle =\lambda$ is satisfied,
then we know that $X: M\rightarrow \mathbb{R}^{n+1}$ is a $\lambda$-hypersurface. Therefore,
the last two terms in \eqref{eq:1004} vanish for any $h$ and any $y$ from \eqref{eq:1001}
and \eqref{eq:1003} of Lemma \ref{lemma 10}. Therefore $X: M\rightarrow \mathbb{R}^{n+1}$ is a critical point of $\mathcal{F}_{0,1}$.
\end{proof}

\begin{corollary}\label{corollary 5.1}
$X: M\rightarrow \mathbb{R}^{n+1}$ is a critical point of $\mathcal{F}_{X_s,t_s}(s)$ if and only if
$M$ is the critical point of $\mathcal{F}$-functional with respect to fixed $X_0$ and $t_0$.
\end{corollary}

\vskip 5mm

\section{The second variation of $\mathcal{F}$-functional}

\noindent
In this section, we shall give the second variation formula of $\mathcal{F}$-functional.
\begin{theorem}\label{theorem 5}
Let $X: M\rightarrow \mathbb{R}^{n+1}$ be a critical point of the functional $\mathcal{F}(s)=\mathcal{F}_{X_s,t_s}(s)$.  The second variation formula of $\mathcal{F}(s)$ for $X_0=0$ and $t_0=1$ is given by

\begin{equation*}
\aligned
 &\ \ \ \ (4\pi)^{\frac{n}{2}}\mathcal{F}^{''}(0)\\
 &=-\int_M fLf e^{-\frac{|X|^2}{2}}d\mu+\int_M \bigl(-|y|^2+\langle X,y\rangle ^2\bigl)e^{-\frac{|X|^2}{2}}d\mu\\
&\ \ \ +\int_M \biggl\{2\langle N,y\rangle +(n+1-|X|^2)\lambda h-2hH-2\lambda\langle X,y\rangle \biggl\}f e^{-\frac{|X|^2}{2}}d\mu\\
&\ \ \ +\int_M \biggl\{(|X|^2-n-1)\langle X,y\rangle \biggl\}h e^{-\frac{|X|^2}{2}}d\mu\\
&\ \ \ +\int_M \biggl\{\frac{n^2+2n}{4}-\frac{n+2}{2}|X|^2+\frac{|X|^4}{4}+\frac{3\lambda}{4}(\lambda-H)\biggl\}h^2 e^{-\frac{|X|^2}{2}}d\mu,
\endaligned
\end{equation*}
where the operator $L$ is defined by
$$L=\mathcal{L}+S+1-\lambda^2.$$
\end{theorem}
\begin{proof}
Let
\begin{equation*}
\aligned
I(s)&=(4\pi t_s)^{-\frac{n}{2}}\int_M -(H_s+\langle \frac{X(s)-X_s}{t_s}, N(s)\rangle )fe^{-\frac{|X(s)-X_s|^2}{2t_s}}d\mu_s\\
&\ \ \ +(4\pi t_0)^{-\frac{n}{2}}\sqrt{\frac{t_0}{t_s}}\int_M \lambda f\langle N(s),N\rangle e^{-\frac{|X-X_0|^2}{2t_0}}d\mu,
\endaligned
\end{equation*}

\begin{equation*}
\aligned
II(s)&=(4\pi t_s)^{-\frac{n}{2}}\int_M \langle \frac{X(s)-X_s}{t_s},y\rangle e^{-\frac{|X(s)-X_s|^2}{2t_s}}d\mu_s\\
&\ \ \ +(4\pi t_0)^{-\frac{n}{2}}\sqrt{\frac{t_0}{t_s}}\int_M \lambda \langle -y,N\rangle e^{-\frac{|X-X_0|^2}{2t_0}}d\mu,
\endaligned
\end{equation*}

\begin{equation*}
\aligned
III(s)&=(4\pi t_s)^{-\frac{n}{2}}\int_M (-\frac{n}{2t_s}+\frac{|X(s)-X_s|^2}{2t_s^2})he^{-\frac{|X(s)-X_s|^2}{2t_s}}d\mu_s\\
&\ \ \ +(4\pi t_0)^{-\frac{n}{2}}\sqrt{\frac{t_0}{t_s}}\int_M -\frac{h\lambda}{2t_s}\langle X(s)-X_s,N\rangle e^{-\frac{|X-X_0|^2}{2t_0}}d\mu,
\endaligned
\end{equation*}

\noindent we have

\begin{equation*}
 \mathcal{F}^{'}(s)=I(s)+II(s)+III(s), \ \   \mathcal{F}^{''}(s)=I^{'}(s)+II^{'}(s)+III^{'}(s),
\end{equation*}

\begin{equation*}
\aligned
I^{'}(s)&=(4\pi t_s)^{-\frac{n}{2}}\int_M \frac{nh}{2t_s}(H_s+\langle \frac{X(s)-X_s}{t_s},N(s)\rangle) fe^{-\frac{|X(s)-X_s|^2}{2t_s}}d\mu_s\\
& \ \ \  +(4\pi t_s)^{-\frac{n}{2}}\int_M -\biggl(\frac{dH_s}{ds}
   +\langle \frac{\frac{\partial X(s)}{\partial s}-\frac{\partial X_s}{\partial s}}{t_s},N(s)\rangle -\langle \frac{X(s)-X_s}{t_s^2},N(s)\rangle h\\
&\ \ \ \ \ \ \ \ \ \ \ \ \ \ \ \ \ \ \ \ \ \ \ \ \ \   +\langle \frac{X(s)-X_s}{t_s},\frac{dN(s)}{ds}\rangle \biggl)fe^{-\frac{|X(s)-X_s|^2}{2t_s}}d\mu_s\\
&\ \ \ +(4\pi t_s)^{-\frac{n}{2}}\int_M -(H_s+\langle \frac{X(s)-X_s}{t_s}, N(s)\rangle )f^{'}e^{-\frac{|X(s)-X_s|^2}{2t_s}}d\mu_s\\
&\ \ \ +(4\pi t_s)^{-\frac{n}{2}}\int_M (H_s+\langle \frac{X(s)-X_s}{t_s},N(s)\rangle )\times\\
&\ \ \ \ \ \ \ \ \ \ \ \ \ \ \ \ \ \ \ \ \ \ \ (\langle \frac{X(s)-X_s}{t_s},
   \frac{\partial X(s)}{\partial s}-\frac{\partial X_s}{\partial s}\rangle +H_sf)fe^{-\frac{|X(s)-X_s|^2}{2t_s}}d\mu_s\\
&\ \ \  +(4\pi t_s)^{-\frac{n}{2}}\int_M-(H_s+\langle \frac{X(s)-X_s}{t_s},N(s)\rangle )f\frac{|X(s)-X_s|^2}{2t_s^2}h\\
&  \ \ \ \ \ \ \ \ \ \ \ \ \ \ \ \ \ \ \ \ \ \ \ \times e^{-\frac{|X(s)-X_s|^2}{2t_s}}d\mu_s\\
   & \ \ \ +(4\pi t_0)^{-\frac{n}{2}}\sqrt{\frac{t_0}{t_s}}\int_M -\frac{h}{2t_s}\lambda \langle N(s),N\rangle fe^{-\frac{|X-X_0|^2}{2t_0}}d\mu\\
   & \ \ \   +(4\pi t_0)^{-\frac{n}{2}}\sqrt{\frac{t_0}{t_s}}\int_M \lambda f^{'}\langle N(s),N\rangle e^{-\frac{|X-X_0|^2}{2t_0}}d\mu\\
&  \ \ \  +(4\pi t_0)^{-\frac{n}{2}}\sqrt{\frac{t_0}{t_s}}\int_M \lambda f\langle \frac{dN(s)}{ds},N\rangle e^{-\frac{|X-X_0|^2}{2t_0}}d\mu,
\endaligned
\end{equation*}

\begin{equation*}
\aligned
&\ \ II^{'}(s)\\
&=(4\pi t_s)^{-\frac{n}{2}}(-\frac{nh}{2t_s})\int_M\langle \frac{X(s)-X_s}{t_s},y\rangle e^{-\frac{|X(s)-X_s|^2}{2t_s}}d\mu_s\\
&  \ \ \ +(4\pi t_s)^{-\frac{n}{2}}\int_M (\langle \frac{\frac{\partial X(s)}{\partial s}-\frac{\partial X_s}{\partial s}}{t_s},y\rangle
  -\langle \frac{X(s)-X_s}{t_s^2},y\rangle h)e^{-\frac{|X(s)-X_s|^2}{2t_s}}d\mu_s\\
& \ \  \ +(4\pi t_s)^{-\frac{n}{2}}\int_M\langle \frac{X(s)-X_s}{t_s},y^{'}\rangle e^{-\frac{|X(s)-X_s|^2}{2t_s}}d\mu_s \\
&  \ \ \  +(4\pi t_s)^{-\frac{n}{2}}\int_M \langle \frac{X(s)-X_s}{t_s},y\rangle \biggl(-\langle \frac{X(s)-X_s}{t_s},
  \frac{\partial X(s)}{\partial s}-\frac{\partial X_s}{\partial s}\rangle\
-H_sf\biggl)e^{-\frac{|X(s)-X_s|^2}{2t_s}}d\mu_s\\
&\ \ \   +(4\pi t_s)^{-\frac{n}{2}}\int_M \langle \frac{X(s)-X_s}{t_s},y\rangle \frac{|X(s)-X_s|^2}{2t_s^2}he^{-\frac{|X(s)-X_s|^2}{2t_s}}d\mu_s\\
& \ \ \ +(4\pi t_0)^{-\frac{n}{2}}\sqrt{\frac{t_0}{t_s}}(-\frac{h}{2t_s})\int_M -\lambda\langle N,y\rangle e^{-\frac{|X-X_0|^2}{2t_0}}d\mu
 +(4\pi t_0)^{-\frac{n}{2}}\sqrt{\frac{t_0}{t_s}}\int_M -\lambda\langle N, y^{'}\rangle e^{-\frac{|X-X_0|^2}{2t_0}}d\mu,
\endaligned
\end{equation*}

\begin{equation*}
\aligned
&\ \ III^{'}(s)\\
&=(4\pi t_s)^{-\frac{n}{2}}(-\frac{nh}{2t_s})\int_M (-\frac{n}{2t_s}+\frac{|X(s)-X_s|^2}{2t_s^2})he^{-\frac{|X(s)-X_s|^2}{2t_s}}d\mu_s\\
&  \ \ \ +(4\pi t_s)^{-\frac{n}{2}}\int_M (\frac{nh}{2t_s^2}-\frac{|X(s)-X_s|^2}{t_s^3}h
+\frac{\langle X(s)-X_s,\frac{\partial X(s)}{\partial s}-\frac{\partial X_s}{\partial s}\rangle }{t_s^2})\times\\
&  \ \ \ \ \ \ \ \ \ \ \ \ \ \ \ \ \ \ \ \ \ \ he^{-\frac{|X(s)-X_s|^2}{2t_s}}d\mu_s\\
& \ \ \ +(4\pi t_s)^{-\frac{n}{2}}\int_M(-\frac{n}{2t_s}+\frac{|X(s)-X_s|^2}{2t_s^2})h^{'}e^{-\frac{|X(s)-X_s|^2}{2t_s}}d\mu_s\\
&\ \ \   +(4\pi t_s)^{-\frac{n}{2}}\int_M (-\frac{n}{2t_s}+\frac{|X(s)-X_s|^2}{2t_s^2})h(-H_sf\\
&  \ \ \ \ \ \ \ \ \ \ \ \ \ \ \ \ \ \ \ \ \
  -\langle \frac{X(s)-X_s}{t_s},\frac{\partial X(s)}{\partial s}-\frac{\partial X_s}{\partial s}\rangle )e^{-\frac{|X(s)-X_s|^2}{2t_s}}d\mu_s\\
&\ \ \   +(4\pi t_s)^{-\frac{n}{2}}\int_M (-\frac{n}{2t_s}+\frac{|X(s)-X_s|^2}{2t_s^2})h
 \frac{|X(s)-X_s|^2}{2t_s^2}he^{-\frac{|X(s)-X_s|^2}{2t_s}}d\mu_s\\
& \ \ \ +(4\pi t_0)^{-\frac{n}{2}}\sqrt{\frac{t_0}{t_s}}(-\frac{h}{2t_s})
   \int_M -\frac{h}{2t_s}\lambda \langle X(s)-X_s, N\rangle e^{-\frac{|X-X_0|^2}{2t_0}}d\mu\\
&\ \ \  +(4\pi t_0)^{-\frac{n}{2}}\sqrt{\frac{t_0}{t_s}}\int_M -\frac{h^{'}\lambda}{2t_s}\langle X(s)-X_s,N\rangle e^{-\frac{|X-X_0|^2}{2t_0}}d\mu
\\
&\ \ \  +(4\pi t_0)^{-\frac{n}{2}}\sqrt{\frac{t_0}{t_s}}\int_M (\frac{h}{2t_s^2}\langle X(s)-X_s,N\rangle \lambda h\\
&  \ \ \ \ \ \ \ \ \ \ \ \ \ \ \ \ \ \ \ \ \ \ \ \ \ \
 -\frac{1}{2t_s}\langle \frac{\partial X(s)}{\partial s}-\frac{\partial X_s}{\partial s},N\rangle \lambda h)e^{-\frac{|X-X_0|^2}{2t_0}}d\mu.
\endaligned
\end{equation*}

Since $X: M\rightarrow \mathbb{R}^{n+1}$ is a critical point, we get

\begin{equation*}
H+\langle \frac{X-X_0}{t_0},N\rangle =\lambda,
\end{equation*}

\begin{equation*}
\int_M (n+\lambda\langle X-X_0,N\rangle -\frac{|X-X_0|^2}{t_0})e^{-\frac{|X-X_0|^2}{2t_0}}d\mu=0,
\end{equation*}

\begin{equation*}
\int_M (\lambda\langle N,a\rangle -\langle \frac{X-X_0}{t_0},a\rangle )e^{-\frac{|X-X_0|^2}{2t_0}}d\mu=0.
\end{equation*}
On the other hand,
\begin{equation*}
H^{'}=\Delta f+Sf,\ \ N^{'}=-\nabla f.
\end{equation*}
Using of the above equations and letting $s=0$, we obtain
\begin{equation*}
\aligned
&\ \ \ \ (4\pi t_0)^{\frac{n}{2}}\mathcal{F}^{''}(0)\\
 &=\int_M -fLf e^{-\frac{|X-X_0|^2}{2t_0}}d\mu\\
&\ \ \ +\int_M \biggl(\frac{2}{t_0}\langle N,y\rangle +\frac{2h}{t_0}\langle \frac{X-X_0}{t_0},N\rangle +\frac{n-1}{t_0}\lambda h\\
 &\ \ \ \ \ \ \ \ \ \ \ \ -\frac{|X-X_0|^2}{t_0^2}\lambda h-2\lambda\langle \frac{X-X_0}{t_0},y\rangle \biggl)fe^{-\frac{|X-X_0|^2}{2t_0}}d\mu\\
&\ \ \ +\int_M \biggl(-\frac{n+2}{t_0}\langle \frac{X-X_0}{t_0},y\rangle +\frac{\lambda}{t_0}\langle N,y\rangle \\
&\ \ \ \ \ \ \ \ \ \ \ \ +\langle \frac{X-X_0}{t_0},y\rangle
  \frac{|X-X_0|^2}{t_0^2}\biggl)he^{-\frac{|X-X_0|^2}{2t_0}}d\mu\\
&\ \ \ +\int_M \biggl(\frac{n^2}{4t_0^2}+\frac{n}{2t_0^2}-\frac{n+2}{2t_0^3}|X-X_0|^2+\frac{|X-X_0|^4}{4t_0^4}\\
 &\ \ \ \ \ \ \ \ \ \ +\frac{3\lambda}{4t_0}\langle \frac{X-X_0}{t_0},N\rangle \biggl)h^2e^{-\frac{|X-X_0|^2}{2t_0}}d\mu\\
&\ \ \ +\int_M \biggl(-\frac{1}{t_0}\langle y,y\rangle +\langle \frac{X-X_0}{t_0},y\rangle ^2\biggl)e^{-\frac{|X-X_0|^2}{2t_0}}d\mu,
\endaligned
\end{equation*}
where the operator $L$ is defined by $L=\Delta+S+\frac{1}{t_0}-\langle \frac{X-X_0}{t_0},\nabla\rangle -\lambda^2$.
When $t_0=1$, $X_0=0$, then $L=\mathcal{L}+S+1-\lambda^2$.
\begin{equation*}
\aligned
&\ \ \ \ (4\pi)^{\frac{n}{2}}\mathcal{F}^{''}(0)\\
&=\int_M -fLf e^{-\frac{|X-|^2}{2}}d\mu\\
&\ \ \ +\int_M \biggl(2\langle N,y\rangle +2\lambda h+(n-1)\lambda h-2hH\\
&\ \ \ \ \ \ \ \ \ \ \ -|X|^2\lambda h-2\lambda\langle X,y\rangle \biggl)f e^{-\frac{|X|^2}{2}}d\mu\\
&\ \ \ +\int_M \biggl(\lambda\langle N,y\rangle -(n+2)\langle X,y\rangle +\langle X,y\rangle |X|^2\biggl)h e^{-\frac{|X|^2}{2}}d\mu\\
\endaligned
\end{equation*}

\begin{equation*}
\aligned
&\ \ \ +\int_M \biggl(\frac{n^2+2n}{4}-\frac{n+2}{2}|X|^2+\frac{|X|^4}{4}+\frac{3\lambda}{4}\langle X,N\rangle \biggl)h^2 e^{-\frac{|X|^2}{2}}d\mu\\
&\ \ \ +\int_M -(|y|^2-\langle X,y\rangle ^2)e^{-\frac{|X|^2}{2}}d\mu\\
&=\int_M -fLf e^{-\frac{|X|^2}{2}}d\mu\\
&\ \ \ +\int_M \biggl[2\langle N,y\rangle +(n+1-|X|^2)\lambda h-2hH-2\lambda\langle X,y\rangle \biggl]f e^{-\frac{|X-|^2}{2}}d\mu\\
&\ \ \ +\int_M \biggl\{(|X|^2-n-1)\langle X,y\rangle \biggl\}h e^{-\frac{|X|^2}{2}}d\mu\\
&\ \ \ +\int_M \biggl(\frac{n^2+2n}{4}-\frac{n+2}{2}|X|^2+\frac{|X|^4}{4}+\frac{3\lambda}{4}(\lambda-H)\biggl)h^2 e^{-\frac{|X|^2}{2}}d\mu\\
&\ \ \ +\int_M (-|y|^2+\langle X,y\rangle ^2)e^{-\frac{|X|^2}{2}}d\mu.
\endaligned
\end{equation*}
\end{proof}

\begin{definition}
One calls that a critical point $X: M\rightarrow \mathbb{R}^{n+1}$ of the $\mathcal{F}$-functional $\mathcal{F}_{X_s,t_s}(s)$
is $\mathcal{F}$-stable if, for every
normal variation $fN$, there exist variations of $X_0$ and $t_0$ such that  $\mathcal{F^{\prime\prime}}_{X_0,t_0}(0)\geq 0$;

\noindent
One calls that a critical point $X: M\rightarrow \mathbb{R}^{n+1}$ of the $\mathcal{F}$-functional $\mathcal{F}_{X_s,t_s}(s)$
is $\mathcal{F}$-unstable if there exist a
normal variation $fN$ such that for all variations of $X_0$ and $t_0$,   $\mathcal{F^{\prime\prime}}_{X_0,t_0}(0)< 0$.
\end{definition}

\begin{theorem}\label{theorem 6.2}
If  $r\leq \sqrt{n}$ or  $r>\sqrt{n+1}$, the $n$-dimensional round sphere $X: {S}^n(r)\rightarrow \mathbb{R}^{n+1}$
is  $\mathcal{F}$-stable;
If  $ \sqrt{n}< r\leq\sqrt{n+1}$,  the $n$-dimensional round sphere $X: {S}^n(r)\rightarrow \mathbb{R}^{n+1}$
is  $\mathcal{F}$-unstable.
\end{theorem}
\begin{proof}
For the sphere ${S}^n(r)$, we have
$$
X=-rN,\ \ H=\frac{n}{r},\ \ S=\frac{H^2}{n}=\frac{n}{r^2},\ \ \lambda=H-r=\frac{n}{r}-r
$$
and
\begin{equation}
L f=\mathcal{L}f+(S+1-\lambda^2)f=\Delta f+(\frac{n}{r^2}+1-\lambda^2)f.
\end{equation}
Since we know that  eigenvalues $\mu_k$ of $\Delta$ on the sphere ${S}^n(r)$ are given by
\begin{equation}
\mu_k=\frac{k^2+(n-1)k}{r^2},
\end{equation}
and  constant functions are eigenfunctions corresponding to eigenvalue $\mu_0=0$.
For any constant vector $z\in \mathbb{R}^{n+1}$, we get
\begin{equation}
-\Delta\langle z,N\rangle =\Delta\langle z,\frac{X}{r}\rangle =\langle z,\frac{1}{r}HN\rangle =\frac{n}{r^2}\langle z,N\rangle ,
\end{equation}
that is, $\langle z,N\rangle $ is an eigenfunction of $\Delta$ corresponding to the first eigenvalue $\mu_1=\frac{n}{r^2}$.
Hence, for any  normal variation with the variation vector field $fN$,
we can choose a real number $a\in \mathbb{R}$ and a constant vector $z\in \mathbb{R}^{n+1}$ such that
\begin{equation}
f=f_0+a+\langle z,N\rangle ,
\end{equation}
and  $f_0$ is in the space spanned by all eigenfunctions corresponding to eigenvalues $\mu_k$ $(k\geq2)$ of $\Delta$ on ${S}^n(r)$.
Using Lemma \ref{lemma 10}, we get
\begin{equation}\label{eq:500}
\aligned
&\ \ \ \ (4\pi)^{\frac{n}{2}}e^{\frac{r^2}{2}}\mathcal{F}^{''}(0)\\
&=\int_{S^n(r)} -(f_0+a+\langle z,N\rangle )L(f_0+a+\langle z,N\rangle ) d\mu\\
&\ \ \ +\int_{S^n(r)} [2\langle N,y\rangle +(n+1-r^2)\lambda h-2\frac{n}{r}h+2\lambda\langle rN,y\rangle ] (f_0+a+\langle z,N\rangle ) d\mu\\
&\ \ \ +\int_{S^n(r)} \lambda\langle N,y\rangle (r^2-n-1)h d\mu\\
&\ \ \ +\int_{S^n(r)} (\frac{n^2+2n}{4}-\frac{n+2}{2}r^2+\frac{r^4}{4}+\frac{3}{4}r^2-\frac{3}{4}n)h^2 d\mu\\
&\ \ \ +\int_{S^n(r)} (-|y|^2+\langle X,y\rangle ^2)d\mu\\
&\geq\int_{S^n(r)} \biggl\{(\frac{n+2}{r^2}-1+\lambda^2)f_0^2-(\frac{n}{r^2}+1-\lambda^2)a^2+(\lambda^2-1)\langle z,N\rangle ^2\biggl\}d\mu\\
&\ \ \ +\int_{S^n(r)}\biggl\{2(1+\lambda r)\langle N,y\rangle \langle N,z\rangle +[(n+1-r^2)\lambda-2\frac{n}{r}]ah\biggl\}d\mu\\
&\ \ \ +\int_{S^n(r)}\frac{1}{4}[r^4-(2n+1)r^2+n(n-1)]h^2d\mu\\
&\ \ \ +\int_{S^n(r)}(-|y|^2+\langle X,y\rangle ^2)d\mu.
\endaligned
\end{equation}
From Lemma \ref{lemma 10}, we have
\begin{equation}\label{eq:501}
\int_{S^n(r)}(-|y|^2+\langle X,y\rangle ^2)d\mu=-\int_{S^n(r)} (1+\lambda r)\langle N,y\rangle ^2d\mu.
\end{equation}
Putting \eqref{eq:501} and $\lambda=\frac{n}{r}-r$ into \eqref{eq:500}, we obtain

\begin{equation}\label{eq:502}
\aligned
&\ \ \ \ (4\pi)^{\frac{n}{2}}e^{\frac{r^2}{2}}\mathcal{F}^{''}(0)\\
&\geq\int_{S^n(r)} \frac{1}{r^2}\biggl\{(r^2-n-\frac{1}{2})^2+\frac{7}{4}\biggl\}f_0^2d\mu\\
&\ \ \ +\int_{S^n(r)}[r^4-(2n+1)r^2+n(n-1)](\frac{a}{r}+\frac{h}{2})^2d\mu\\
&\ \ \ +\int_{S^n(r)}\frac{1}{r^2}[r^4-(2n+1)r^2+n^2]\langle z,N\rangle ^2d\mu\\
&\ \ \ +\int_{S^n(r)} 2(1+n-r^2)\langle N,y\rangle \langle N,z\rangle d\mu\\
&\ \ \ +\int_{S^n(r)} -(1+n-r^2)\langle N,y\rangle ^2d\mu.
\endaligned
\end{equation}
 If we choose $h=-\frac{2a}{r}$ and $y=kz$, then we have

\begin{equation}\label{eq:504}
\aligned
&\ \ \ \ (4\pi)^{\frac{n}{2}}e^{\frac{r^2}{2}}\mathcal{F}^{''}(0)\\
&\geq\int_{S^n(r)} \frac{1}{r^2}\biggl\{(r^2-n-\frac{1}{2})^2+\frac{7}{4}\biggl\}f_0^2d\mu\\
&\ \ \ +\int_{S^n(r)} \biggl\{\lambda^2-1+2(1+\lambda r)k-(1+\lambda r)k^2\biggl\}\langle z,N\rangle ^2d\mu\\
&=\int_{S^n(r)} \frac{1}{r^2}\biggl\{(r^2-n-\frac{1}{2})^2+\frac{7}{4}\biggl\}f_0^2d\mu\\
&\ \ \ +\int_{S^n(r)} \biggl\{\lambda^2+\lambda r-(1+\lambda r)(1-k)^2\biggl\}\langle z,N\rangle ^2d\mu.
\endaligned
\end{equation}
We next consider three cases:

\vskip 3mm
{\bf Case 1: $r\leq \sqrt{n}$}
\vskip3mm

\noindent
In this case, $\lambda\geq 0$. Taking $k=1$, then we get
$$\mathcal{F}^{''}(0)\geq 0.$$

\vskip3mm
{\bf Case 2: $r\geq\frac{1+\sqrt{1+4n}}{2}$}.
\vskip3mm

\noindent
In this case, $\lambda\leq -1$.  Taking $k=2$, we can get

$$\mathcal{F}^{''}(0)\geq 0.$$

\vskip3mm
{\bf Case 3: $\sqrt{n+1}< r< \frac{1+\sqrt{1+4n}}{2}$}.
\vskip3mm

\noindent
In this case, $-1< \lambda<0$, $1+\lambda r<0$, we can take $k$ such that
$(1-k)^2\geq\frac{\lambda(\lambda+r)}{1+\lambda r}$, then we have

$$\mathcal{F}^{''}(0)\geq 0.
$$
Thus, if $r\leq \sqrt{n}$ or  $r>\sqrt{n+1}$, the $n$-dimensional round sphere $X: {S}^n(r)\rightarrow \mathbb{R}^{n+1}$
is  $\mathcal{F}$-stable;

\vskip3mm
\noindent
If  $ \sqrt{n}< r\leq\sqrt{n+1}$,  the $n$-dimensional round sphere $X: {S}^n(r)\rightarrow \mathbb{R}^{n+1}$
is  $\mathcal{F}$-unstable. In fact,
in this case, $-1< \lambda< 0$, $1+\lambda r\geq0$.  We can choose $f$ such that $f_0=0$, then we have
\begin{equation}
\aligned
(4\pi)^{\frac{n}{2}}e^{\frac{r^2}{2}}\mathcal{F}^{''}(0)
&=\int_{S^n(r)} (\lambda^2-1)\langle z,N\rangle ^2d\mu\\
&\ \ \ +\int_{S^n(r)} 2(1+\lambda r)\langle N,y\rangle \langle N,z\rangle d\mu\\
&\ \ \ +\int_{S^n(r)}-(1+\lambda r)\langle N,y\rangle ^2d\mu\\
&=(\lambda^2+\lambda r)\int_{S^n(r)}\langle z,N\rangle ^2d\mu\\
&\ \ \ -(1+\lambda r)\int_{S^n(r)} (\langle z,N\rangle -\langle y,N\rangle )^2d\mu\\
&< 0.
\endaligned
\end{equation}
 This completes the proof of Theorem \ref{theorem 6.2}.
\end{proof}

\noindent
According to Theorem \ref{theorem 6.2}, we would like to propose the following:

\vskip2mm
\noindent
{\bf Problem 3.1}.  Is it possible to prove that
spheres $S^n(r)$ with $r\leq \sqrt n$ or $r>\sqrt {n+1}$ are  the only $\mathcal F$-stable
compact $\lambda$-hypersurfaces?
\begin{remark}   Colding and Minicozzi \cite{[CZ]} have proved that the sphere $S^n(\sqrt n)$ is
the only $\mathcal F$-stable compact self-shrinkers. In order to prove this result,
the property that the mean curvature $H$ is an eigenfunction of $L$-operator plays a very important role.
But for $\lambda$-hypersurfaces, the mean curvature $H$ is not an eigenfunction of $L$-operator in general.
\end{remark}

\vskip 5mm

\section{The weak stability of the weighted area functional for  weighted volume-preserving variations}

\noindent
Define
\begin{equation}
\mathcal{T}(s)=(4\pi t_s)^{-\frac{n}{2}}\int_M e^{-\frac{|X(s)-X_s|^2}{2t_s}}d\mu_s.
\end{equation}
We compute the first and the second variation formulas of the general $\mathcal{T}$-functional for weighted volume-preserving variations with fixed $X_0$ and $t_0$. By a direct calculation, we have
\begin{equation*}
\mathcal{T}^{'}(s)=(4\pi t_s)^{-\frac{n}{2}}\int_M -(H_s+\langle \frac{X(s)-X_s}{t_s}, N(s)\rangle )fe^{-\frac{|X(s)-X_s|^2}{2t_s}}d\mu_s,
\end{equation*}

\begin{equation*}
\aligned
\mathcal{T}^{''}(s)
&=(4\pi t_s)^{-\frac{n}{2}}\int_M -(H_s+\langle \frac{X(s)-X_s}{t_s}, N(s)\rangle )f^{'}e^{-\frac{|X(s)-X_s|^2}{2t_s}}d\mu_s\\
&\ \ \ +(4\pi t_s)^{-\frac{n}{2}}\int_M (H_s+\langle \frac{X(s)-X_s}{t_s},N(s)\rangle )\times\\
&\ \ \ \ \ \ \ \ \ \ \ \ \ \ \ \ \ \ \ \ \ \ \ (\langle \frac{X(s)-X_s}{t_s},
   \frac{\partial X(s)}{\partial s}\rangle +H_sf)fe^{-\frac{|X(s)-X_s|^2}{2t_s}}d\mu_s\\
   &    \ \ \  +(4\pi t_s)^{-\frac{n}{2}}\int_M -\biggl(\frac{dH_s}{ds}
   +\langle \frac{\frac{\partial X(s)}{\partial s}}{t_s},N(s)\rangle \\
   &  \ \ \ \ \ \ \ \ \ \ \ \ \ \ \ \ \ \ \ \ \ \ \ +\langle \frac{X(s)-X_s}{t_s},\frac{dN(s)}{ds}\rangle \biggl)fe^{-\frac{|X(s)-X_s|^2}{2t_s}}d\mu_s.
\endaligned
\end{equation*}

\begin{lemma}
$$\int_Mf^{'}(0)e^{-\frac{|X-X_0|^2}{2t_0}}d\mu=0.$$
\end{lemma}
\begin{proof}
Since $V(t)=\int_M\langle X(t)-X_0,N\rangle e^{-\frac{|X-X_0|^2}{2t_0}}d\mu=V(0)$ for any $t$, we have
$$\int_Mf(t)\langle N(t),N\rangle e^{-\frac{|X-X_0|^2}{2t_0}}d\mu=0.
$$
Hence, we get
\begin{equation*}
\aligned
0&=\frac{d}{dt}|_{t=0}\int_Mf(t)\langle N(t),N\rangle e^{-\frac{|X-X_0|^2}{2t_0}}d\mu\\
 &=\int_M f^{'}(0)e^{-\frac{|X-X_0|^2}{2t_0}}d\mu.
\endaligned
\end{equation*}
\end{proof}

\noindent
Since $M$ is a critical point of $\mathcal{T}(s)$, we have
$$
H+\langle \frac{X-X_0}{t_0},N\rangle =\lambda.
$$
On the other hand, we have
\begin{equation}
H^{'}=\Delta f+Sf, \ \ N^{'}=-\nabla f.
\end{equation}
Then for $t_0=1$ and $X_0=0$, the second variation formula becomes
\begin{equation*}
 (4\pi)^{\frac{n}{2}}\mathcal{T}^{''}(0)=\int_M -f\bigl(\mathcal{L}f+(S+1-\lambda^2)f\bigl) e^{-\frac{|X|^2}{2}}d\mu.
\end{equation*}

\begin{theorem}\label{theorem 7.1}
Let $X: M\rightarrow \mathbb{R}^{n+1}$ be a critical point of the functional $\mathcal{T}(s)$ for the weighted volume-preserving variations with fixed $X_0=0$ and $t_0=1$.
The second variation formula of $\mathcal{T}(s)$  is given by
\begin{equation}
\aligned
&\ \ \ \ (4\pi)^{\frac{n}{2}}\mathcal{T}^{''}(0)=\int_M -f\bigl(\mathcal{L}f+(S+1-\lambda^2)f\bigl) e^{-\frac{|X|^2}{2}}d\mu.
\endaligned
\end{equation}
\end{theorem}

\begin{definition}
A critical point $X: M\rightarrow \mathbb{R}^{n+1}$ of the  functional $\mathcal{T}(s)$
is  called weakly stable if, for any  weighted volume-preserving normal variation,
  $\mathcal{T^{\prime\prime}}(0)\geq 0$;

\noindent
A critical point $X: M\rightarrow \mathbb{R}^{n+1}$ of the  functional $\mathcal{T}(s)$
is  called weakly unstable if there exists  a weighted volume-preserving normal variation,
  such that $\mathcal{T^{\prime\prime}}(0)< 0$.
\end{definition}

\begin{theorem}\label{theorem 7.2}
If  $r\leq \frac{-1+\sqrt{1+4n}}{2}$ or  $r\geq \frac{1+\sqrt{1+4n}}{2}$, the $n$-dimensional round sphere $X: {S}^n(r)\rightarrow \mathbb{R}^{n+1}$
is   weakly stable;
If  $ \frac{-1+\sqrt{1+4n}}{2}< r < \frac{1+\sqrt{1+4n}}{2}$,  the $n$-dimensional round sphere $X: {S}^n(r)\rightarrow \mathbb{R}^{n+1}$
is  weakly unstable.
\end{theorem}
\begin{proof}
For the sphere ${S}^n(r)$, we have
$$
X=-rN,\ \ H=\frac{n}{r},\ \ S=\frac{n}{r^2},\ \ \lambda=H-r=\frac{n}{r}-r
$$
and
\begin{equation}
Lf=\mathcal{L}f+(S+1-\lambda^2)f=\Delta f+(\frac{n}{r^2}+1-\lambda^2)f.
\end{equation}
Since we know that  eigenvalues $\mu_k$ of $\Delta$ on the sphere ${S}^n(r)$ are given by
\begin{equation}
\mu_k=\frac{k^2+(n-1)k}{r^2},
\end{equation}
and  constant functions are eigenfunctions corresponding to eigenvalue $\mu_0=0$. For any constant vector $z\in \mathbb{R}^{n+1}$, we get
\begin{equation}
-\Delta\langle z,N\rangle =\frac{n}{r^2}\langle z,N\rangle ,
\end{equation}
that is, $\langle z,N\rangle $ is an eigenfunction of $\Delta$ corresponding to the first eigenvalue $\mu_1=\frac{n}{r^2}$.
Hence, for any weighted volume-preserving normal variation with the variation vector field $fN$ satisfying
$$
\int_{S^n(r)} fe^{-\frac{r^2}{2}}d\mu=0,
$$
we can choose  a constant vector $z\in \mathbb{R}^{n+1}$ such that

\begin{equation}
f=f_0+\langle z,N\rangle ,
\end{equation}
and  $f_0$ is in the space spanned by all eigenfunctions corresponding to eigenvalues $\mu_k$ $(k\geq2)$ of $\Delta$ on ${S}^n(r)$.
By making use of Theorem  \ref{theorem 7.1}, we have
\begin{equation}\label{eq:700}
\aligned
&\ \ \ \ (4\pi)^{\frac{n}{2}}e^{\frac{r^2}{2}}\mathcal{T}^{''}(0)\\
&=\int_{S^n(r)} -(f_0+\langle z,N\rangle )L(f_0+\langle z,N\rangle ) d\mu\\
&\geq\int_{S^n(r)} \biggl\{(\frac{n+2}{r^2}-1+\lambda^2)f_0^2+(\lambda^2-1)\langle z,N\rangle ^2\biggl\}d\mu.
\endaligned
\end{equation}
According to  $\lambda=\frac{n}{r}-r$, we obtain

\begin{equation*}
\aligned
&(4\pi)^{\frac{n}{2}}e^{\frac{r^2}{2}}\mathcal{T}^{''}(0)\\
&\geq\int_{S^n(r)} \frac{1}{r^2}\bigl\{(r^2-n-\frac{1}{2})^2+\frac{7}{4}\bigl\}f_0^2d\mu
 +\int_{S^n(r)} (\frac nr-r-1)(\frac nr-r+1)\langle z,N\rangle ^2d\mu\geq 0
\endaligned
\end{equation*}
if
$$
r\leq \dfrac{-1+\sqrt{4n+1}}2  \ \ \ {\rm or} \ \ \ r\geq \dfrac{1+\sqrt{4n+1}}2 .
$$
Thus,  the $n$-dimensional round sphere $X: {S}^n(r)\rightarrow \mathbb{R}^{n+1}$
is  weakly stable.

\noindent
If
$$
 \dfrac{-1+\sqrt{4n+1}}2 < r<\dfrac{1+\sqrt{4n+1}}2,
$$
 choosing $f=<z, N>$,
 we have
 $$
\int_{S^n(r)} fe^{-\frac{r^2}{2}}d\mu=0.
$$
Hence,  there exists a weighted volume-preserving normal variation
with the variation vector filed $fN$ such that
\begin{equation*}
\aligned
(4\pi)^{\frac{n}{2}}e^{\frac{r^2}{2}}\mathcal{T}^{''}(0)
=\int_{S^n(r)} (\frac nr-r-1)(\frac nr-r+1)\langle z,N\rangle ^2d\mu< 0.
\endaligned
\end{equation*}
Thus, the $n$-dimensional round sphere $X: {S}^n(r)\rightarrow \mathbb{R}^{n+1}$
is  weakly unstable.
It finishes the proof.
\end{proof}

\begin{remark}
From Theorem \ref{theorem 6.2} and Theorem \ref{theorem 7.2}, we know the $\mathcal{F}$-stability and
the weak stability are different. The $\mathcal{F}$-stability is a weaker notation than
the weak stability.
\end{remark}
\begin{remark}
Is it possible to prove that
spheres $S^n(r)$ with $r\leq \frac{-1+\sqrt{1+4n}}{2}$ or  $r\geq \frac{1+\sqrt{1+4n}}{2}$ are  the only  weak stable
compact $\lambda$-hypersurfaces?
\end{remark}
\vskip 5mm

\section{Properness and polynomial  area growth for $\lambda$-hypersurfaces}

\noindent
For $n$-dimensional complete and non-compact Riemannian manifolds with nonnegative Ricci curvature, the well-known
theorem of Bishop and Gromov says that  geodesic balls have at most polynomial area growth:
$$
{\rm Area} (B_r(x_0))\leq Cr^n.
$$
For $n$-dimensional complete and non-compact gradient shrinking Ricci soliton, Cao and Zhou \cite{CZ} have proved
geodesic balls have at most polynomial area growth.
For self-shrinkers, Ding and Xin \cite{[DX1]} proved that any complete non-compact properly immersed self-shrinker in the Euclidean space has polynomial area growth. X. Cheng and Zhou \cite{[CZ]} showed that any complete immersed self-shrinker with polynomial area growth in the Euclidean space is proper. Hence any complete immersed self-shrinker is proper if and only if it has polynomial area growth.

\noindent It is our purposes in this section to study the area growth for
$\lambda$-hypersurfaces.
First of all,  we study the equivalence of properness and polynomial area growth for  $\lambda$-hypersurfaces.
If $X: M\rightarrow \mathbb{R}^{n+1}$ is an $n$-dimensional hypersurface in $\mathbb{R}^{n+1}$,
we say $M$ has polynomial area growth if there exist constant $C$ and $d$ such that for all $r\geq 1$,
\begin{equation}
{\rm Area}(B_r(0)\cap X(M))=\int_{B_r(0)\cap X(M)}d\mu\leq Cr^d,
\end{equation}
where $B_r(0)$ is a round ball in $\mathbb{R}^{n+1}$ with radius $r$ and centered at the origin.

\begin{theorem}\label{theorem 9.1}
Let $X: M\rightarrow \mathbb{R}^{n+1}$ be a complete and non-compact properly immersed
$\lambda$-hypersurface in the Euclidean space $\mathbb{R}^{n+1}$.
Then, there is a positive constant $C$ such that for $r\geq1$,
\begin{equation}
{\rm Area}(B_r(0)\cap X(M))=\int_{B_r(0)\cap X(M)}d\mu\leq Cr^{n+\frac{\lambda^2}2-2\beta-\frac{\inf H^2}2},
\end{equation}
where  $\beta=\frac{1}{4}\inf(\lambda-H)^2$.
\end{theorem}
\begin{proof}
Since $X: M\rightarrow \mathbb{R}^{n+1}$
is  a complete and non-compact properly immersed
$\lambda$-hypersurface in the Euclidean space $\mathbb{R}^{n+1}$, we
have
$$
\langle X, N\rangle +H=\lambda.
$$
Defining $f=\frac{|{X}|^2}{4}$, we have
\begin{equation}
f-|{\nabla}f|^2=\frac{|{X}|^2}{4}-\frac{|{X}^{T}|^2}{4}
=\frac{|{X}^{\perp}|^2}{4}=\frac14(\lambda-{H})^2,
\end{equation}
\begin{equation}
\aligned
{\Delta} f&=\frac{1}{2}(n+{H}\langle N,{X}\rangle) \\
        &=\frac{1}{2}(n+\lambda \langle N,{X}\rangle-\langle N,{X}\rangle^2)\\
        &=\frac{1}{2}n+\frac{\lambda^2}4-\frac{H^2}4-f+|{\nabla}f|^2.
\endaligned
\end{equation}
Hence, we obtain
\begin{equation}
|{\nabla}{(f-\beta)}|^2\leq (f-\beta),
\end{equation}
\begin{equation}
{\Delta}{(f-\beta)}-|{\nabla}{(f-\beta)}|^2+ {(f-\beta)}\leq (\frac{n}{2}+\frac{\lambda^2}4-\beta -\frac{\inf H^2}4 ).
\end{equation}
Since the  immersion $X$ is proper, we know that $\overline{f}=f-\beta$ is proper.
Applying Theorem 2.1 of X. Cheng and Zhou \cite{[CZ]} to $\overline{f}=f-\beta$
with $k=(\frac{n}{2}+\frac{\lambda^2}4-\beta-\frac{\inf H^2}4)$, we obtain
 \begin{equation}
{\rm Area}(B_r(0)\cap X(M))=\int_{B_r(0)\cap X(M)}d\mu\leq Cr^{n+\frac{\lambda^2}2-2\beta-\frac{\inf H^2}2},
\end{equation}
where  $\beta=\frac{1}{4}\inf(\lambda-H)^2$ and $C$ is  a constant.
\end{proof}

\begin{remark}
The estimate in Theorem \ref{theorem 9.1} is the best possible because the cylinders $S^k(r_0)\times \mathbb{R}^{n-k}$
satisfy the equality.
\end{remark}
\begin{remark}
By making use of the same assertions as  in X. Cheng and Zhou \cite{[CZ]} for self-shrinkers,  we can prove
the weighted area of  a complete and non-compact properly immersed
$\lambda$-hypersurface in the Euclidean space $\mathbb{R}^{n+1}$
is bounded.
\end{remark}

\noindent
By making use of  to the same assertions  as in X. Cheng and Zhou \cite{[CZ]} for self-shrinkers,
we can prove the following theorem. We will leave it  for readers.
\begin{theorem}
If  $X: M\rightarrow \mathbb{R}^{n+1}$ is  an $n$-dimensional complete immersed $\lambda$-hypersurface with polynomial area  growth, then
$X: M\rightarrow \mathbb{R}^{n+1}$ is proper.
\end{theorem}

\vskip 2mm

\section{A lower bound growth of the area for  $\lambda$-hypersurfaces}

\noindent
For $n$-dimensional complete and non-compact Riemannian manifolds with nonnegative Ricci curvature, the well-known
theorem of Calabi and Yau says that geodesic balls have at least linear area growth:
$$
{\rm Area} (B_r(x_0))\geq Cr.
$$
Cao and Zhu \cite{CZu} have proved that  $n$-dimensional complete and non-compact gradient shrinking Ricci soliton
must have infinite volume. Furthermore, Munteanu and Wang \cite{MW} have proved that
areas of geodesic balls for  $n$-dimensional complete and non-compact gradient shrinking Ricci soliton
 has at least linear growth. For self-shrinkers, Li and Wei \cite{[LW2]} proved that any complete and non-compact proper self-shrinker has at least linear area growth.

\noindent In this section, we study the lower bound growth of the area for $\lambda$-hypersurfaces.
The following lemmas play a very important role in order to prove our results.

\begin{lemma}\label{lemma:4-20-1}
Let $X: M\rightarrow \mathbb{R}^{n+1}$ be an $n$-dimensional complete noncompact proper $\lambda$-hypersurface, then
there exist constants $C_1(n,\lambda)$ and $c(n,\lambda)$ such that for all $t\geq C_1(n,\lambda)$,
\begin{equation}
{\rm Area}(B_{t+1}(0)\cap X(M))-{\rm Area}(B_t(0)\cap X(M))\leq c(n,\lambda)\frac{{\rm Area}(B_t(0)\cap X(M))}{t}
\end{equation}
and
\begin{equation}
{\rm Area}(B_{t+1}(0)\cap X(M))\leq 2{\rm Area}(B_{t}(0)\cap X(M)).
\end{equation}
\end{lemma}
\begin{proof}
Since $X: M\rightarrow \mathbb{R}^{n+1}$ is a complete $\lambda$-hypersurface, one has
\begin{equation}\label{eq:4-20-1}
\frac{1}{2}\Delta |X|^2=n+H\langle N,X\rangle=n+H\lambda-H^2.
\end{equation}
Integrating \eqref{eq:4-20-1} over $B_r(0)\cap X(M)$, we obtain
\begin{equation}\label{eq:4-20-5}
\aligned
&\ \ \ \ n{\rm Area}(B_{r}(0)\cap X(M))+\int_{B_r(0)\cap X(M)}H\lambda d\mu-\int_{B_r(0)\cap X(M)}H^2d\mu\\
 &=\frac{1}{2}\int_{B_r(0)\cap X(M)}\triangle |X|^2d\mu\\
&=\frac{1}{2}\int_{\partial(B_r(0)\cap X(M))}\nabla |X|^2\cdot\frac{\nabla \rho}{|\nabla\rho|}d\sigma\\
&=\int_{\partial(B_r(0)\cap X(M))}|X^{T}|d\sigma\\
&=\int_{\partial(B_r(0)\cap X(M))}\frac{|X|^2-(\lambda-H)^2}{|X^{T}|}d\sigma\\
&=r({\rm Area}(B_{r}(0)\cap X(M)))^{'}-\int_{\partial(B_r(0)\cap X(M))}\frac{(\lambda-H)^2}{|X^{T}|}d\sigma,
\endaligned
\end{equation}
where $\rho(x):=|X(x)|$, $\nabla \rho=\frac{X^T}{|X|}$.
Here we used, from the co-area formula,
\begin{equation}\label{eq:4-20-5-1}
\bigl({\rm Area}(B_{r}(0)\cap X(M))\bigl)^{'}=r\int_{\partial(B_r(0)\cap X(M))}\frac{1}{|X^{T}|}d\sigma.
\end{equation}
Hence, we obtain
\begin{equation}\label{eq:4-20-2}
\aligned
&\ \ \ \ (n+\frac{\lambda^2}{4}){\rm Area}(B_{r}(0)\cap X(M))-r({\rm Area}(B_{r}(0)\cap X(M)))^{'}\\
&=\int_{B_r(0)\cap X(M)}(H-\frac{\lambda}{2})^2d\mu-\int_{\partial(B_r(0)\cap X(M))}\frac{(\lambda-H)^2}{|X^{T}|}d\sigma,
\endaligned
\end{equation}
From \eqref{eq:4-20-5-1},   $(H-\lambda)^2=\langle N,{X}\rangle^2\leq |X|^2=r^2$ on $\partial(B_r(0)\cap X(M))$
and  \eqref{eq:4-20-2},
we conclude
\begin{equation}
\int_{B_{r}(0)\cap X(M)}(H-\frac{\lambda}{2})^2d\mu\leq (n+\frac{\lambda^2}{4}){\rm Area}(B_{r}(0)\cap X(M)).
\end{equation}
Furthermore, we have
\begin{equation}\label{eq:4-20-6}
\aligned
\int_{B_{r}(0)\cap X(M)}(H-\lambda)^2d\mu & \leq \int_{B_{r}(0)\cap X(M)}2\bigl[(H-\frac{\lambda}{2})^2
  +\frac{\lambda^2}{4}\bigl]d\mu\\
&\leq (2n+\lambda^2){\rm Area}(B_{r}(0)\cap X(M)),
\endaligned
\end{equation}
\begin{equation}\label{eq:12-25-1}
\aligned
\int_{B_{r}(0)\cap X(M)}H^2d\mu & \leq \int_{B_{r}(0)\cap X(M)}2\bigl[(H-\frac{\lambda}{2})^2
  +\frac{\lambda^2}{4}\bigl]d\mu\\
&\leq (2n+\lambda^2){\rm Area}(B_{r}(0)\cap X(M)).
\endaligned
\end{equation}
\eqref{eq:4-20-2} implies that
\begin{equation}\label{eq:4-20-3}
\aligned
&\ \ \ \ \bigl(r^{-n-\frac{\lambda^2}{4}}{\rm Area}(B_{r}(0)\cap X(M))\bigl)^{'}\\
&=r^{-n-1-\frac{\lambda^2}{4}}\biggl(r\bigl({\rm Area}(B_{r}(0)\cap X(M))\bigl)^{'}-(n+\frac{\lambda^2}{4})
{\rm Area}(B_{r}(0)\cap X(M))\biggl)\\
&=r^{-n-1-\frac{\lambda^2}{4}}\int_{\partial(B_r(0)\cap X(M))}\frac{(H-\lambda)^2}{|X^{T}|}d\sigma
 -r^{-n-1-\frac{\lambda^2}{4}}\int_{B_r(0)\cap X(M)}(H-\frac{\lambda}{2})^2d\mu.
\endaligned
\end{equation}
Integrating \eqref{eq:4-20-3} from $r_2$ to $r_1$ ($r_1>r_2$), one has
\begin{equation}\label{eq:4-20-4}
\aligned
&\ \ \ \ r_1^{-n-\frac{\lambda^2}{4}}{\rm Area}(B_{r_1}(0)\cap X(M))-
         r_2^{-n-\frac{\lambda^2}{4}}{\rm Area}(B_{r_2}(0)\cap X(M))\\
&=r_1^{-n-2-\frac{\lambda^2}{4}}\int_{B_{r_1}(0)\cap X(M)}(H-\lambda)^2d\mu
  -r_2^{-n-2-\frac{\lambda^2}{4}}\int_{B_{r_2}(0)\cap X(M)}(H-\lambda)^2d\mu\\
&\ \ \ \
  +(n+2+\frac{\lambda^2}{4})\int_{r_2}^{r_1}s^{-n-3-\frac{\lambda^2}{4}}(\int_{B_{s}(0)\cap X(M)}(H-\lambda)^2d\mu)ds\\
&\ \ \ \ -\int_{r_2}^{r_1}s^{-n-1-\frac{\lambda^2}{4}}(\int_{B_{s}(0)\cap X(M)}(H-\frac{\lambda}{2})^2d\mu)ds\\
&\leq \bigl(r_1^{-n-2-\frac{\lambda^2}{4}}
 +r_2^{-n-2-\frac{\lambda^2}{4}}\bigl)\int_{B_{r_1}(0)\cap X(M)}(H-\lambda)^2d\mu.
\endaligned
\end{equation}
Here we used
$$
\biggl(\int_{B_{r}(0)\cap X(M)}(H-\lambda)^2d\mu\biggl)^{'} =r\int_{\partial(B_r(0)\cap X(M))}\frac{(H-\lambda)^2}{|X^{T}|}d\sigma
$$
and ${\rm Area}(B_{r}(0)\cap X(M))$ is non-decreasing in $r$ from  \eqref{eq:4-20-5-1}.
Combining \eqref{eq:4-20-4} with  \eqref{eq:4-20-6}, we have
\begin{equation}\label{eq:4-20-7}
\aligned
&\ \ \ \
\frac{{\rm Area}(B_{r_1}(0)\cap X(M))}{r_1^{n+\frac{\lambda^2}{4}}}
-\frac{{\rm Area}(B_{r_2}(0)\cap X(M))}{r_2^{n+\frac{\lambda^2}{4}}}\\
&\leq (2n+\lambda^2)\bigl(\frac{1}{r_1^{n+2+\frac{\lambda^2}{4}}}
+ \frac{1}{r_2^{n+2+\frac{\lambda^2}{4}}}\bigl){\rm Area}(B_{r_1}(0)\cap X(M)).
\endaligned
\end{equation}
Putting $r_1=t+1$, $r_2=t>0$, we get
\begin{equation}\label{eq:4-20-8}
\aligned
 &\biggl(1-\frac{2(2n+\lambda^2)(t+1)^{n+\frac{\lambda^2}{4}}}
                                     {t^{n+2+\frac{\lambda^2}{4}}}\biggl){\rm Area}(B_{{t+1}}(0)\cap X(M)) \\
                                     & \leq {\rm Area}(B_{{t}}(0)\cap X(M)) (\frac{t+1}{t})^{n+\frac{\lambda^2}{4}}.
                                     \endaligned
\end{equation}
For $t$ sufficiently large, one has, from \eqref{eq:4-20-8},
\begin{equation}\label{eq:4-20-9}
\aligned
&\ \ \ \
{\rm Area}(B_{t+1}(0)\cap X(M))-{\rm Area}(B_{t}(0)\cap X(M))\\
&\leq {\rm Area}(B_{t}(0)\cap X(M))\biggl ((1+\frac{1}{t})^n-1+\frac{C(t+1)^{2n+\lambda^2}{4}}{t^{2n+2+\lambda^2}}\biggl),
\endaligned
\end{equation}
where ${C}$ is constant only depended on $n$, $\lambda$.
Therefore,  there exists some constant $C_1(n,\lambda)$ such that for all $t\geq C_1(n,\lambda)$,
\begin{equation}
\aligned
&{\rm Area}(B_{t+1}(0)\cap X(M))-{\rm Area}(B_{t}(0)\cap X(M))\\
&\leq c(n,\lambda)\frac{{\rm Area}(B_{t}(0)\cap X(M))}{t},
\endaligned
\end{equation}
\begin{equation}
{\rm Area}(B_{t+1}(0)\cap X(M))\leq 2{\rm Area}(B_{t}(0)\cap X(M)),
\end{equation}
where $c(n,\lambda)$ depends only on $n$ and $\lambda$. This completes the proof of Lemma \ref{lemma:4-20-1}.

\end{proof}

\vskip3mm
\noindent
The following Logarithmic Sobolev inequality for hypersurfaces in Euclidean space is due to Ecker \cite{[E]},
\begin{lemma}\label{lemma:4-20-2}
Let $X: M\rightarrow \mathbb{R}^{n+1}$ be an $n$-dimensional hypersurface with measure $d\mu$. Then the following inequality
\begin{equation}\label{eq:4-20-10}
\aligned
&\ \ \ \
\int_M f^2(\ln f^2)e^{-\frac{|X|^2}{2}}d\mu-\int_M f^2e^{-\frac{|X|^2}{2}}d\mu \ln (\int_M f^2e^{-\frac{|X|^2}{2}}d\mu)\\
&\leq 2\int_M|\nabla f|^2e^{-\frac{|X|^2}{2}}d\mu+\frac{1}{2}\int_M|H+\langle X,N\rangle|^2f^2e^{-\frac{|X|^2}{2}}d\mu\\
&\ \ \  +C_1(n)\int_Mf^2e^{-\frac{|X|^2}{2}}d\mu,
\endaligned
\end{equation}
\begin{equation}\label{eq:12-25-2}
\aligned
&\ \ \ \
\int_M f^2(\ln f^2)d\mu-\int_M f^2d\mu \ln (\int_M f^2d\mu)\\
&\leq 2\int_M|\nabla f|^2d\mu+\frac{1}{2}\int_M|H|^2f^2d\mu+C_2(n)\int_Mf^2d\mu
\endaligned
\end{equation}
hold for any nonnegative function $f$ for which all integrals are well-defined and finite, where $C_1(n)$ and $C_2(n)$ are positive constants depending on $n$.
\end{lemma}

\begin{corollary}\label{corollary:4-27-1}
For  an $n$-dimensional $\lambda$-hypersurface $X: M\rightarrow \mathbb{R}^{n+1}$,  we have  the following inequality
\begin{equation}
\int_M f^2(\ln f)e^{-\frac{|X|^2}{2}}d\mu\leq \int_M|\nabla f|^2e^{-\frac{|X|^2}{2}}d\mu
+\frac{1}{2}C_1(n)+\frac{1}{4}\lambda^2
\end{equation}
 for any nonnegative function $f$ which satisfies
\begin{equation}
\int_M f^2e^{-\frac{|X|^2}{2}}d\mu=1.
\end{equation}
\end{corollary}

\begin{lemma}$($\cite{[LW2]}$)$\label{lemma:4-27-1}
Let $X: M\rightarrow \mathbb{R}^{n+1}$ be a complete properly immersed hypersurface.
For any $x_0\in M$, $r\leq 1$, if $|H|\leq\frac{C}{r}$ in $B_r(X(x_0))\cap X(M)$ for some constant $C>0$. Then
\begin{equation}
\text{Area}(B_{r}(X(x_0))\cap X(M))\geq \kappa r^n,
\end{equation}
where $\kappa=\omega_n e^{-C}$.
\end{lemma}

\begin{lemma}\label{lemma:4-21-1}
If $X: M\rightarrow \mathbb{R}^{n+1}$ is an $n$-dimensional complete and non-compact proper $\lambda$-hypersurface, then it has infinite area.
\end{lemma}

\begin{proof}
Let
\begin{equation*}
\Omega(k_1,k_2)=\{x\in M: 2^{k_1-\frac{1}{2}}\leq \rho(x)\leq 2^{k_2-\frac{1}{2}}\},
\end{equation*}
\begin{equation*}
A(k_1,k_2)=\text{Area}(X(\Omega(k_1,k_2))),
\end{equation*}
where $\rho(x)=|X(x)|$. Since $X: M\rightarrow \mathbb{R}^{n+1}$ is a complete and non-compact proper immersion, $X(M)$ can not
be contained in a compact Euclidean ball. Then, for $k$ large enough, $\Omega(k, k+1)$ contains at least $2^{2k-1}$ disjoint balls
\begin{equation*}
B_r(x_i)=\{ x\in M: \rho_{x_i}(x)<2^{-\frac{1}{2}}r\}, \ \ x_i\in M,\ r=2^{-k}
\end{equation*}
where $\rho_{x_i}(x)=|X(x)-X(x_i)|$. Since, in $\Omega(k,k+1)$,
\begin{equation}
|H|\leq |H-\lambda|+|\lambda|=|\langle X,N\rangle|+|\lambda|\leq |X|+|\lambda|\leq 2^{k}\sqrt2+|\lambda|
\leq \frac{\sqrt2+|\lambda|}{r},\end{equation}
by using of Lemma \ref{lemma:4-27-1}, we get
\begin{equation}\label{eq:4-27-2}
A(k,k+1)\geq \kappa_1 2^{2k-1-kn},
\end{equation}
with $\kappa_1=\omega_n e^{-(\sqrt{2}+|\lambda|)2^{-\frac{1}{2}}}2^{-\frac{n}{2}}$.

{\bf Claim}: If ${\text Area}(X(M))<\infty$, then, for every $\varepsilon>0$, there exists a large constant $k_0>0$ such that,
\begin{equation}\label{eq:4-27-1}
A(k_1,k_2)\leq \varepsilon \ \ \ {\rm and}\ \ \ A(k_1,k_2)\leq 2^{4n}A(k_1+2,k_2-2), \ \ \ {\rm if }\ \ k_2>k_1>k_0.
\end{equation}

\noindent
In fact, we may choose $K>0$ sufficiently large such that  $k_1\approx \frac{K}{2}$, $k_2\approx\frac{3K}{2}$.
Assume  \eqref{eq:4-27-1} does not hold, that is,
\begin{equation*}
A(k_1,k_2)\geq 2^{4n}A(k_1+2,k_2-2).
\end{equation*}
If
\begin{equation*}
A(k_1+2,k_2-2)\leq 2^{4n}A(k_1+4,k_2-4),
\end{equation*}
then we complete the proof of the claim. Otherwise, we can repeat the procedure for  $j$ times, we have
\begin{equation*}
A(k_1,k_2)\geq 2^{4nj}A(k_1+2j,k_2-2j).
\end{equation*}
When $j\approx\frac{K}{4}$, we have from \eqref{eq:4-27-2}
\begin{equation*}
\text{Area}(X(M))\geq A(k_1,k_2)\geq 2^{nK}A(K,K+1)\geq \kappa_1 2^{2K-1}.
\end{equation*}
Thus, \eqref{eq:4-27-1}  must hold for some $k_2>k_1$
because $\text{Area}(M)<\infty$.
Hence for any $\varepsilon>0$, we can choose $k_1$ and $k_2\approx 3k_1$ such that \eqref{eq:4-27-1} holds.

\noindent
We  define a smooth cut-off function $\psi(t)$ by
\begin{equation}
 \psi(t)= {\begin{cases}
      \ 1, & \ \ 2^{k_1+\frac{3}{2}}\leq t\leq 2^{k_2-\frac{5}{2}},\\
      \ 0,& \ \ {\rm outside}\ [2^{k_1-\frac{1}{2}}, 2^{k_2-\frac{1}{2}}].\\
     \end{cases}}
     \ \ \ \ \ 0\leq \psi(t)\leq 1, \ \ \ \ |\psi^{'}(t)|\leq 1.
  \end{equation}
Moreover, $\psi(t)$ can be defined in such a way that
\begin{equation}
0\leq \psi^{\prime}(t)\leq\frac{c_1}{2^{k_1-\frac{1}{2}}}, \ \ \ t\in [2^{k_1-\frac{1}{2}}, 2^{k_1+\frac{3}{2}}],
\end{equation}
\begin{equation}
-\frac{c_2}{2^{k_2-\frac{1}{2}}}\leq \psi^{\prime}(t)\leq 0, \ \ \ t\in [2^{k_2-\frac{5}{2}}, 2^{k_2-\frac{1}{2}}],
\end{equation}
for some positive constants $c_1$ and $c_2$.

\noindent Letting
\begin{equation}
f(x)=e^{L+\frac{|X|^2}{4}}\psi(\rho(x)),
\end{equation}
we choose the constant $L$ satisfying
\begin{equation}
1=\int_M f^2e^{-\frac{|X|^2}{2}}d\mu=e^{2L}\int_{\Omega(k_1,k_2)}\psi^2(\rho(x))d\mu.
\end{equation}
We obtain from Corollary \ref{corollary:4-27-1}, $t\ln t\geq-\frac{1}{e}$ for $0\leq t\leq 1$, $|\nabla \rho|\leq 1$
and $\psi^{'}(\rho(x))\leq 0$ in $\Omega(k_1+2,k_2)$ that
\begin{equation}
\aligned
\frac{1}{2}C_1(n)+\frac{1}{4}\lambda^2&\geq\int_{\Omega(k_1,k_2)}e^{2L}\psi^2(L+\frac{|X|^2}{4}+\ln \psi)d\mu\\
              &\ \ \ -\int_{\Omega(k_1,k_2)}e^{2L}|\psi^{'}\nabla \rho+\psi\frac{X^T}{2}|^2d\mu\\
              &\geq\int_{\Omega(k_1,k_2)}e^{2L}\psi^2(L+\frac{|X|^2}{4}+\ln \psi)d\mu\\
              &\ \ -\int_{\Omega(k_1,k_2)}e^{2L}|\psi^{'}|^2d\mu-\frac{1}{4}\int_{\Omega(k_1,k_2)}e^{2L}\psi^2|X|^2d\mu\\
              &\ \ -\frac{1}{2}\int_{\Omega(k_1,k_2)}e^{2L}\psi^{'}\psi\frac{|X^T|^2}{|X|}d\mu\\
              &\geq L+\int_{\Omega(k_1,k_2)}e^{2L}\psi^2\ln \psi d\mu-\int_{\Omega(k_1,k_2)}e^{2L}|\psi^{'}|^2 d\mu\\
              &\ \ -\frac{1}{2}\int_{\Omega(k_1,k_1+2)}e^{2L}\psi^{'}\psi |X|d\mu\\
              &\geq L-(\frac{1}{2e}+1)e^{2L}A(k_1,k_2)-2c_1e^{2L}A(k_1,k_1+2).
\endaligned
\end{equation}
Therefore,  it follows from \eqref{eq:4-27-1} that
\begin{equation}\label{eq:4-27-3}
\aligned
\frac{1}{2}C_1(n)+\frac{1}{4}\lambda^2&\geq  L-(\frac{1}{2e}+1+2c_1)e^{2L}2^{4n}A(k_1+2,k_2-2)\\
              & \geq  L-(\frac{1}{2e}+1+2c_1)e^{2L}2^{4n}\int_{\Omega(k_1,k_2)}\psi^2(\rho(x))d\mu\\
              &= L-(\frac{1}{2e}+1+2c_1)2^{4n}.
\endaligned
\end{equation}
On the other hand, we have, from \eqref{eq:4-27-1} and definition of $f(x)$,
\begin{equation}
1\leq e^{2L}\varepsilon.
\end{equation}
Letting $\varepsilon>0$ sufficiently small, then $L$ can be arbitrary large, which  contradicts
 \eqref{eq:4-27-3}. Hence,  $M$ has infinite area.
\end{proof}

\begin{theorem}\label{theorem:4-21-1}
Let $X: M\rightarrow \mathbb{R}^{n+1}$ be an $n$-dimensional complete proper $\lambda$-hypersurface.
Then, for any $p\in M$, there exists a constant $C>0$ such that
$$
\text{Area}(B_{r}(X(x_0))\cap X(M))\geq Cr,
$$
for all $r>1$.
\end{theorem}

\begin{proof}
We can choose $r_0>0$ such that ${\rm Area}(B_{r}(0)\cap X(M))>0$ for $r\geq r_0$. It is sufficient to prove there exists a constant $C>0$ such that
\begin{equation}\label{eq:4-21-1}
{\rm Area}(B_{r}(0)\cap X(M))\geq Cr
\end{equation}
holds for all $r\geq r_0$. In fact, if \eqref{eq:4-21-1} holds, then for any $x_0\in M$ and $r>|X(x_0)|$,
\begin{equation}
B_r(X(x_0))\supset B_{r-|X(x_0)|}(0),
\end{equation}
and
\begin{equation}\label{eq:4-21-2}
{\rm Area}(B_{r}(X(x_0))\cap X(M)) \geq {\rm Area}(B_{r-|X(x_0)|}(0)\cap X(M))\geq \frac{C}{2}r,
\end{equation}
for $r\geq 2|X(x_0)|$.

\noindent
We next prove \eqref{eq:4-21-1} by contradiction. Assume for any $\varepsilon>0$, there exists $r\geq r_0$ such that
\begin{equation}\label{eq:4-21-3}
{\rm Area}(B_{r}(0)\cap X(M))\leq \varepsilon r.
\end{equation}
Without loss of generality, we assume $r\in \mathbb{N}$ and consider a set:
\begin{equation*}
D:=\{k\in \mathbb{N}: {\rm Area}(B_{t}(0)\cap X(M))\leq 2\varepsilon t {\rm \ for\ any \  integer}\  t  \ {\rm satisfying } \ r\leq t\leq k\}.
\end{equation*}
Next, we will show that  $k\in D$  for any integer $k$ satisfying $k\geq r$.
For $t\geq r_0$, we define a function $u$ by

\begin{equation}
 u(x)= {\begin{cases}
      \ t+2-\rho(x), & \ \ {\rm in}  \ \  B_{t+2}(0)\cap X(M)\setminus B_{t+1}(0)\cap X(M),\\
      \ 1,& \ \ {\rm in} \ \ B_{t+1}(0)\cap X(M)\setminus B_{t}(0)\cap X(M),\\
      \ \rho(x)-(t-1), & \ \ {\rm in}  \ \  B_{t}(0)\cap X(M)\setminus B_{t-1}(0)\cap X(M),\\
      \ 0, & \ \ {\rm otherwise}.\\
     \end{cases}}
  \end{equation}
Using Lemma \ref{lemma:4-20-2}, $|\nabla \rho|\leq 1$ and $t\ln t\geq -\frac{1}{e}$ for $0\leq t\leq 1$, we have
\begin{equation}\label{eq:4-21-5}
\aligned
&\ \
-\dfrac{1}{2}\int_M u^2d\mu\ln \bigl\{\bigl({\rm Area}(B_{t+2}(0)\cap X(M))-{\rm Area}(B_{t-1}(0)\cap X(M))\bigl)\bigl\}\\
&\leq C_0 \biggl({\rm Area}(B_{t+2}(0)\cap X(M))-{\rm Area}(B_{t-1}(0)\cap X(M))\biggl)\\
&\ \ +\dfrac{1}{4}\biggl(\int_{B_{t+2}(0)\cap X(M)}H^2d\mu-\int_{B_{t-1}(0)\cap X(M)}H^2d\mu\biggl),
\endaligned
\end{equation}
where $C_0=1+\frac{1}{2e}+\frac{1}{2}C_2(n)$, $C_2(n)$ is the constant of Lemma \ref{lemma:4-20-2}.

\noindent For all $t\geq C_1(n,\lambda)+1$, we have from Lemma \ref{lemma:4-20-1}
\begin{equation}\label{eq:4-21-4}
\aligned
&\ \ {\rm Area}(B_{t+2}(0)\cap X(M))-{\rm Area}(B_{t-1}(0)\cap X(M))\\
&\leq c(n,\lambda)\biggl(\frac{{\rm Area}(B_{t+1}(0)\cap X(M))}{t+1}\\
&\ \ \ +\frac{{\rm Area}(B_{t}(0)\cap X(M))}{t}+\frac{{\rm Area}(B_{t-1}(0)\cap X(M))}{t-1}\biggl)\\
&\leq c(n,\lambda)\biggl(\frac{2}{t+1}+\frac{1}{t}+\frac{1}{t}(1+\frac{1}{C_1(n,\lambda)})\biggl){\rm Area}(B_{t}(0)\cap X(M))\\
&\leq C_2(n,\lambda)\frac{{\rm Area}(B_{t}(0)\cap X(M))}{t},
\endaligned
\end{equation}
where $C_2(n,\lambda)$ is constant depended only on $n$ and $\lambda$. Note that we can assume $r\geq C_1(n,\lambda)+1$
for the $r$ satisfying \eqref{eq:4-21-3}. In fact, if for any given $\varepsilon>0$, all the $r$ which satisfies \eqref{eq:4-21-3}
is bounded above by $C_1(n,\lambda)+1$, then ${\rm Area}(B_{r}(0)\cap X(M))\geq C r$ holds for
any $r>C_1(n,\lambda)+1$. Thus, we know that  $M$ has at least linear area growth.
Hence, for any $k\in D$ and any $t$ satisfying $r\leq t\leq k$,  we have
\begin{equation}
{\rm Area}(B_{t+2}(0)\cap X(M))-{\rm Area}(B_{t-1}(0)\cap X(M))\leq 2C_2(n,\lambda)\varepsilon.
\end{equation}
Since
\begin{equation}
\int_M u^2d\mu\geq {\rm Area}(B_{t+1}(0)\cap X(M))-{\rm Area}(B_{t}(0)\cap X(M)),
\end{equation}
holds, if we choose $\varepsilon$ such that $2C_2(n,\lambda)\varepsilon<1$,
from \eqref{eq:4-21-5}, we obtain
\begin{equation}
\aligned
&\ \
({\rm Area}(B_{t+1}(0)\cap X(M))-{\rm Area}(B_{t}(0)\cap X(M)))\ln (2C_2(n,\lambda)\varepsilon)^{-1}\\
&\leq 2C_0\biggl({\rm Area}(B_{t+2}(0)\cap X(M))-{\rm Area}(B_{t-1}(0)\cap X(M))\biggl)\\
&\ \ +\dfrac{1}{2}\biggl(\int_{B_{t+2}(0)\cap X(M)}H^2d\mu-\int_{B_{t-1}(0)\cap X(M)}H^2d\mu\biggl).
\endaligned
\end{equation}
Iterating from $t=r$ to $t=k$ and taking summation on $t$, we infer, from Lemma \ref{lemma:4-20-1} and the equation \eqref{eq:12-25-1} that
\begin{equation}
\aligned
&\ \
({\rm Area}(B_{k+1}(0)\cap X(M))-{\rm Area}(B_{r}(0)\cap X(M)))\ln (2C_2(n,\lambda)\varepsilon)^{-1}\\
&\leq 6C_0{\rm Area}(B_{k+2}(0)\cap X(M))+\dfrac{3}{2}\int_{B_{k+2}(0)\cap X(M)}H^2d\mu\\
&\leq\biggl[6C_0+\dfrac{3}{2}(2n+\lambda^2)\biggl]{\rm Area}(B_{k+2}(0)\cap X(M))\\
&\leq2\biggl[6C_0+\dfrac{3}{2}(2n+\lambda^2)\biggl]{\rm Area}(B_{k+1}(0)\cap X(M)).
\endaligned
\end{equation}
Hence, we get
\begin{equation}\label{eq:4-21-6}
\aligned
&{\rm Area}(B_{k+1}(0)\cap X(M))\\
&\leq \frac{\ln (2C_2(n,\lambda)\varepsilon)^{-1}}
{\ln (2C_2(n,\lambda)\varepsilon)^{-1}-12C_0-3(2n+\lambda^2)}{\rm Area}(B_{r}(0)\cap X(M))\\
&\leq \frac{\ln (2C_2(n,\lambda)\varepsilon)^{-1}}
{\ln (2C_2(n,\lambda)\varepsilon)^{-1}-12C_0-3(2n+\lambda^2)}\varepsilon r.
\endaligned
\end{equation}
We can choose $\varepsilon$ small enough such that
\begin{equation}\label{eq:4-21-8}
\frac{\ln (2C_2(n,\lambda)\varepsilon)^{-1}}
{\ln (2C_2(n,\lambda)\varepsilon)^{-1}-12C_0-3(2n+\lambda^2)}\leq 2.
\end{equation}
Therefore,  it follows from \eqref{eq:4-21-6} that
\begin{equation}\label{eq:4-21-7}
{\rm Area}(B_{k+1}(0)\cap X(M))\leq 2\varepsilon r,
\end{equation}
for any $k\in D$. Since $ k+1\geq r$,  we have,  from \eqref{eq:4-21-7} and the definition of $D$, that $k+1\in D$.
Thus,  by induction, we know that $D$ contains all  of  integers $k\geq r$ and
\begin{equation}
{\rm Area}(B_{k}(0)\cap X(M))\leq 2\varepsilon r,
\end{equation}
for any integer $k\geq r$. This implies that $M$ has finite volume, which contradicts with Lemma \ref{lemma:4-21-1}.
Hence, there exist  constants $C$ and $r_0$ such that ${\rm Area}(B_{r}(0)\cap X(M))\geq C r$ for $r>r_0$.
It  completes the proof of Theorem \ref{theorem:4-21-1}.

\end{proof}

\begin{remark}
The estimate in our theorem  is the best possible because the cylinders $S^{n-1}(r_0)\times \mathbb{R}$
satisfy the equality.
\end{remark}

\vskip 2mm
\noindent
{\bf Acknowledgement}. The authors would like to thank the referees for careful reading of the paper and for the valuable suggestions and comments which make this paper better and more readable.

\end {document}